\documentclass[12pt, a4paper, oneside]{amsart}
\usepackage{amssymb}  
\usepackage{amsmath} 
\usepackage{amsthm}  
\usepackage{amsaddr}

\usepackage[cp1251]{inputenc}
\usepackage[english]{babel}
\usepackage{cite, verbatim, lscape,array, multirow}
\usepackage[justification=centering]{caption}

\newtheoremstyle{noparens}%
{}{}%
{\itshape}{}%
{\bfseries}{.}%
{ }%
{\thmname{#1}\thmnumber{ #2}\thmnote{\rm #3}}

\theoremstyle{noparens}
\newtheorem{theorem}{Theorem}
\newtheorem*{theorem*}{Theorem A}

\newtheorem{lemma}{Lemma}
\numberwithin{lemma}{section}

\makeatletter
\renewenvironment{proof}[1][\proofname] {\par\pushQED{\qed}\normalfont\topsep6\p@\@plus6\p@\relax\trivlist\item[\hskip\labelsep\bfseries#1\@addpunct{.}]\ignorespaces}{\popQED\endtrivlist\@endpefalse}
\makeatother

\numberwithin{equation}{section}

\newcommand{\Aut}{\operatorname{Aut}}
\newcommand{\Out}{\operatorname{Out}}
\newcommand{\Inndiag}{\operatorname{Inndiag}}
\newcommand{\Outdiag}{\operatorname{Outdiag}}

\begin{document}

\title[On the prime graph of a finite group]{On the prime graph of a finite group with unique nonabelian composition factor}

\author{Maria A. Grechkoseeva}
\address{Sobolev Institute of Mathematics\\ Koptyuga 4, Novosibirsk 630090, Russia}

\author{Andrey V. Vasil'ev}
\address{Sobolev Institute of Mathematics\\ Koptyuga 4, Novosibirsk 630090, Russia\\
Novosibirsk State University\\ Pirogova 1, Novosibirsk 630090, Russia}

\email{grechkoseeva@gmail.com, vasand@math.nsc.ru}

\thanks{This work was supported by the Program of Fundamental Scientific Researches of the Siberian Branch of Russian Academy of Sciences, I.1.1, project 0314-2019-0001}

\begin{abstract} We say that finite groups are isospectral if they have the same sets of orders of elements. It is known that every nonsolvable finite group $G$ isospectral to a finite simple group has a unique nonabelian composition factor, that is, the quotient of $G$ by the solvable radical of $G$ is an almost simple group. The main goal of this paper is prove that this almost simple group is a cyclic extension of its socle.

To this end, we consider a general situation when
$G$ is an arbitrary group with unique nonabelian composition factor, not necessarily isospectral to a simple group, and study the prime graph of $G$, where the prime graph of $G$ is the graph whose vertices are the~prime numbers dividing the order of $G$ and two such numbers $r$ and $s$ are adjacent if and only if $r\neq s$ and $G$ has an element of order $rs$. Namely, we establish some sufficient conditions for the prime graph of such a group to have a vertex adjacent to all other vertices.
Besides proving the main result, this allows us to refine a recent result by P.~Cameron and N. Maslova concerning
finite groups almost recognizable by prime graph.

 \end{abstract}

\maketitle

\section{Introduction}

Given a finite group $G$, we denote the set of prime divisors of the order of $G$ by $\pi(G)$. The set of element orders of $G$ is called the spectrum of $G$ and denoted by $\omega(G)$. If $\omega(G)=\omega(H)$, then $G$ and $H$ are said to be isospectral.

Suppose that $G$ is a finite group isospectral to a finite nonabelian simple group $L$. Then
$G$ is either solvable, in which case $L$ is one of $L_3(3)$, $U_3(3)$, $S_4(3)$, or has exactly one nonabelian composition factor (see \cite[Theorem 2]{13Gor.t}). In what follows, we assume that $G$ is not solvable, and so $G$ has a normal series \begin{equation}\label{e:main}1\leqslant K<H\leqslant G,\end{equation}                                                                                                                                                                                                                                                                                                                 where $K$ is the solvable radical of $G$, $H/K$ is a nonabelian simple group and $G/K$ is an~almost simple group with socle $H/K$. Denoting
$H/K$ by $S$, we may identify $G/K$ with a subgroup of $\Aut S$, and then $G/H$ with a subgroup of $\Out S=\Aut S/S$. Observe that $G/H$ is solvable.

If $L$ is sufficiently 'large', more precisely, if $L$ is a classical group of dimension at least 38 or a non-classical group other than $Alt_6$, $Alt_{10}$, $J_2$,  $^3D_4(2)$, then $K=1$ and $H\simeq L$ (see \cite{15VasGr1, 17Sta}). Furthermore, it follows that $G/H$ is cyclic (see \cite{18Gr.t} and the references therein). In general case, $K$ is not always trivial and $H/K$ is not always isomorphic to $L$ but in all known examples, $G/H$ is cyclic. This observation suggests us to conjecture that $G/H$ is always cyclic  and the main goal of this paper is to prove this conjecture.

\begin{theorem}\label{t:cyclic}
Let $L$ be a finite nonabelian simple group and let $G$ be a nonsolvable finite group with $\omega(G)=\omega(L)$. Suppose that $1\leqslant K<H\leqslant G$ is the normal series of $G$ as in \eqref{e:main}.  Then $G/H$ is cyclic. Furthermore, if $H/K$ is a simple group of Lie type other than $L_2(q)$, then $G/H$ does not contain diagonal automorphisms.
\end{theorem}

If $L$ is sporadic or alternating, Theorem \ref{t:cyclic} is a direct consequence of the known description of groups isospectral to $L$. If $L$ is a group of Lie type, the proof has several ingredients. The first is the well-known property of spectra of groups of Lie type stated in Lemma \ref{l:Lie_type} in Section \ref{s:proof}. The second is the nilpotency of the solvable radical of $G$ established in \cite{20YanGrVas}. The third is the following Theorem \ref{t:main} which concerns all finite groups of some specific structure, not only those isospectral to simple groups.

\begin{theorem}\label{t:main} Suppose that a finite group $G$ has a normal series $1\leqslant K<H\leqslant G$, where
$K$ is the solvable radical of $G$, $S=H/K$ is a finite simple group of Lie type, and $G/K\leq \Aut S$. Suppose also that $K$ is nilpotent.
\begin{enumerate}
 \item  If $S\neq L_2(q)$ and $G/H$ contains a diagonal
 automorphism of $S$ of prime order $r$, then $rs\in\omega(G)$ for all $s\in\pi(G)\setminus \{r\}$.
  \item If $G/H$ is not cyclic, then there is $r\in \pi(G/H)$ such that $rs\in\omega(G)$ for all $s\in\pi(G)\setminus \{r\}$.
 \end{enumerate}
\end{theorem}

The set $\omega(G)$ defines the prime graph of $G$ as follows: the vertex set of this is $\pi(G)$ and two primes $r,s\in\pi(G)$ are adjacent if and only if $r\neq s$ and $rs\in\omega(G)$.
The prime graph is also known as the Gruenberg--Kegel graph and we denote it by $GK(G)$. It is not hard to see that Theorem \ref{t:main} states a~property of the graph $GK(G)$ rather than of the~whole set $\omega(G)$. This allows us to apply this theorem to the problem of recognition of simple groups by prime graph. Recently, P. Cameron and N. Maslova \cite{21CamMas} proved several new results relating to this problem. In Theorem \ref{t:poly}, we slightly refine Theorem 1.4 of \cite{21CamMas}.

\begin{theorem}\label{t:poly}
 There exists a function $F(x) = O(x^5)$ such that for each labeled graph $\Gamma$, the following conditions are equivalent:
\begin{enumerate}
 \item there exist infinitely many groups $H$ such that $GK(H) = \Gamma$;
 \item there exist more than $F(|V(\Gamma)|)$ groups $H$ such that $GK(H) = \Gamma$, where $V(\Gamma)$ is the set of the vertices of $\Gamma$.
\end{enumerate}
\end{theorem}

In fact, Theorem 1.4 of \cite{21CamMas} states exactly the same as Theorem \ref{t:poly} but with $x^7$ in place of $x^5$.

\section{Proofs of Theorems \ref{t:cyclic} and \ref{t:main}}\label{s:proof}

We begin this section with notation and preliminary results.
We write $L^\varepsilon_n(q)$ and $E^\varepsilon_6(q)$ assuming that $\varepsilon\in\{+,-\}$, $L_n^+(q)=L_n(q)$, $L_n^-(q)=U_n(q)$, $E_6^+(q)=E_6(q)$, and $E_6^-(q)={}^2E_6(q)$. If $r$ is a prime and $a$ is an integer, then $(a)_r$ is the highest power of $r$ dividing $a$. If $S$ is a group of Lie type, then $\Inndiag S$  is the subgroup of $\Aut S$ generated by inner and diagonal automorphisms, and $\Outdiag S$ is the image of $\Inndiag S$ in $\Out S$. Also we use the terms `field automorphism' and `graph automorphism' of $S$ according to \cite[Definition 2.5.13]{98GorLySol}.

\begin{lemma}\label{l:Lie_type} If $S$ is a finite simple group of Lie type, then for every $r\in\pi(S)$ there is $s\in\pi(S)$ such that $r\neq s$ and $rs\not\in\omega(S)$.
\end{lemma}

\begin{proof}
This follows from \cite{05VasVd.t, 11VasVd.t} (see, for example, \cite[Lemma 2.2]{15Gr}).
\end{proof}

\begin{lemma}\label{l:diag_p} Let $S$ be a finite simple group of Lie type in characteristic $p$. If $r$ divides $|\Outdiag S|$ and $rp\not\in\omega(S)$, then either $S=L_2(q)$, or $S=L_3^\varepsilon(q)$ and $(q-\varepsilon)_3=3$.\end{lemma}

\begin{proof}
This follows, for example, from
\cite[Propositions 3.1 and 3.2]{05VasVd.t}.
\end{proof}

\begin{lemma}\label{l:adj_2} Let $S$ be a finite simple group of Lie type in characteristic $p$. If $r\in\pi(S)$, $r$ is odd and $2r\not\in\omega(S)$, then either a Sylow $r$-subgroup of $S$ is cyclic, or $S=L_2(q)$ and $r=p$, or $S=L_3^\varepsilon(q)$, $p=2$, $r=3$ and $(q-\varepsilon)_3=3$.\end{lemma}

\begin{proof} This follows from the results of
\cite[Sections 3 and 4]{05VasVd.t}
and the cross-characteristic Sylow structure of groups of Lie type \cite[(10-2)]{83GorLy}.
\end{proof}

\begin{lemma}\label{l:field}
Let $S={}^t\Sigma(q)$ be a finite simple group of Lie type, not a Suzuki--Ree group, and let $\varphi$ be a field automorphism of $S$ of prime order $r$. Then $r\cdot\omega(^t\Sigma(q^{1/r}))\subseteq\omega(S\rtimes \langle\varphi\rangle)$.
\end{lemma}

\begin{proof}
 This follows from the Lang--Steinberg theorem \cite[Section 10]{68Ste} (see, for example, \cite[Lemma 2.8]{16Gr.t}).
\end{proof}

\begin{lemma} \label{l:frob} Suppose that $G$ is a finite group, $K$ is a normal subgroup of $G$ and every $g\in G\setminus K$ acts fixed-point-freely on $K$. Then every  odd order Sylow subgroup of $G/K$ is cyclic and a Sylow $2$-subgroup of $G/K$ is cyclic or generalized quaternion.
\end{lemma}

\begin{proof}
This is a well-known property of fixed-point-free automorphisms (see, for example, \cite[Satz 8.7]{67Hup}).
\end{proof}

\begin{proof}[Proof of Theorem \ref{t:main}]
Denote the defining characteristic of $S$ by $p$, $G/K$ by $\overline G$ and $G/H$ by $\widehat G$. As we remarked in the introduction, $\widehat G$ can be regarded as a subgroup of $\Out S$.

Clearly, we may assume that either $\Outdiag S\neq 1$ or $\Out S$ is not cyclic, in particular, we may assume that $S$ is not a Suzuki--Ree group
and so $3\in\pi(S)$.

(i) Suppose that $r\in\pi(\widehat G\cap \Outdiag S)$. Observe that $r\in\pi(S)$ and $r\neq p$. By Lemma \ref{l:diag_p}, it follows that $rp\in\omega(S)$ unless $S=L_3^\varepsilon(q)$, $r=3$ and $(q-\varepsilon)_3=3$. In this case $PGL_3^\varepsilon(q)\leqslant \overline G$, and
since $PGL_3^\varepsilon(q)$ has an element of order $p(q-\varepsilon)$, we see that $rp\in\omega(\overline G)$.

Suppose that $s\in \pi(S)$ and $s\neq p$. If $s\in\pi(\Outdiag S)$, then $rs\in\omega(S)$ since $\Outdiag S$ is abelian. So we may assume that $s\not\in\pi(\Outdiag S)$. The maximal tori of
$\Inndiag S$ are isomorphic to those of the universal version $\tilde S_u$ of $\tilde S$, where $\tilde S=S$ if $S$ is not of type $B_n$ or $C_n$, and $\tilde B_n(q)=C_n(q)$, $\tilde C_n(q)=B_n(q)$ (see \cite[Section 4.4]{85Car}). Since every maximal torus of $\tilde S_u$ contains the center $Z(\tilde S_u)$ of $\tilde S_u$ and $|Z(\tilde S_u)|=|\Outdiag S|$, we see that $\Inndiag S$ includes a maximal torus whose order is divisible by $s|\Outdiag S|$. So
$\overline G$ contains an~abelian subgroup of order $sr$.

Let $s\in\pi(\overline G)\setminus\pi(S)$. Since $s\neq 2,3$ and $s\not \in\pi(\Outdiag S)$, it follows that $G/K$ contains a field  automorphism of $S$ of order $s$. By Lemma \ref{l:field}, we have $s\cdot\omega(S_0)\subseteq \omega(G/K)$, where $S_0$ is a group of the same Lie type as $S$. If $r=2,3$, then it is clear that $r\in\pi(S_0)$. If $r\neq2,3$, then $S=L_n^\varepsilon(q)$, $r$ divides $(n,q-\varepsilon)$ and
$S_0=L_n^\varepsilon(q^{1/s})$. Since $r$ divides $p^{r-1}-\varepsilon^{r-1}$ and $r-1\leqslant n-1$, we see that $r\in\pi(S_0)$.

Let $s\in \pi(K)\setminus\pi(\overline G)$. If $r=2$, then $s$ is adjacent to $r$ in $GK(G)$ by \cite[Proposition 2]{05Vas.t}. So we may assume that $r$ is odd. If $S=E_6^\varepsilon(q)$ or $S=L_n^\varepsilon(q)$ with $n\geqslant 4$, then $S$ includes a~torus of the form $\mathbb Z_{q-\varepsilon}\times \mathbb Z_{q-\varepsilon}$, and hence $S$ includes an elementary abelian group of order $r^2$. If $L=L_3^\varepsilon(q)$, then $PGL_3^\varepsilon(q)\leqslant \overline G$ and so $\overline G$ includes an elementary abelian group of order $r^2$. Now we apply Lemma \ref{l:frob} to conclude that $rs\in\omega(G)$.

(ii) Let $S\neq L_2(q)$. By (i), we may assume that $\widehat G\cap \Outdiag S=1$. Then either $\widehat G$ includes an elementary abelian group of order $2^2$, or $S=O_8^+(q)$ and, up to conjugation in $\Out S$,  $\widehat G$ contains the image of the graph automorphism $\gamma$ of $S$ induced by the symmetry of the Dynkin diagram of order $3$.

In the first case, $S=L_n(q)$, $O_{2n}^+(q)$, or $E_6(q)$,
and we claim that $2$ is adjacent to all odd primes in $GK(G)$. By \cite[Proposition 2]{05Vas.t}, every $s\in\pi(K)\cup\pi(\widehat G)$ is adjacent to $2$. Now let $t\in \pi(S)$ and suppose that $2t\not\in\omega(S)$. Excluding for a while the case when $t=3$, $S=L_3(q)$, $p=2$, $(q-1)_3=3$ and applying Lemma \ref{l:adj_2},
we conclude that a Sylow $t$-subgroup $T$ of $S$ is cyclic, and hence $N_{\overline G}(T)/C_{\overline G}(T)$ is cyclic. On the other hand, by the~Frattini argument, $N_{\overline G}(T)/(N_{\overline G}(T)\cap S)\simeq \widehat G$, and so a Sylow $2$-subgroup of $N_{\overline G}(T)$ is not cyclic. Thus $2\in C_{\overline G}(T)$, and $2t\in \omega(\overline G)$.

Suppose that $t=3$, $S=L_3(q)$, $p=2$, and $(q-1)_3=3$. Since $\widehat G$ includes an elementary abelian group of order $2^2$, it follows that $\overline G$ contains a field automorphism of $S$ of order $2$, and so $6\in\omega(\overline G)$.

Now suppose that $S=O_8^+(q)$ and $\overline G$ contains the graph automorphism $\gamma$.
The centralizer of $\gamma$ in $S$ is isomorphic to $G_2(q)$ \cite[(9-1)]{83GorLy} and so $3s\in\omega(G)$ for all $3\neq s\in G_2(q)$. Since $S$ includes an elementary abelian group of order $9$, we conclude that $3s\in \omega(G)$ for all $s\in \pi(K)\setminus\{3\}$. Also a $2'$-Hall subgroup of $\Out S$ is abelian, and hence $3$ is adjacent to every $s\in\pi(\widehat G)\setminus\{2,3\}$ in $GK(\widehat G)$.. Let $s\in\pi(S)\setminus\{3\}$ and $3s\not\in\omega(S)$. Then $s$ divides $q^2+q+1$ or $q^2-q+1$, therefore, $s\in\pi(G_2(q))$ and, as we remarked, $3s\in\omega(G)$. Thus $3$ is adjacent to all vertices in $GK(G)$.

Let $S=L_2(q)$, where $q=p^l$. We claim that $2$ is adjacent to all odd primes in $GK(G)$. Since $\Out S$ is a direct product of cyclic groups of orders $(2,q-1)$ and $l$, it follows that $p$ is odd, $l$ is even and $\overline G=PGL_2(q)\rtimes \langle \varphi\rangle$, where $\varphi$ is a field automorphism of $S$ of even order. Since $PGL_2(q)$ contains elements of orders $q\pm 1$  and $2p\in\omega(\overline G)$ by Lemma \ref{l:field}, we see that $2$ is adjacent to every odd $s\in\pi(\overline G)$. Let $s\in\pi(K)$ be odd. A Sylow 2-subgroup of $PGL_2(q)$ is dihedral, and so it cannot act fixed-point-freely on a Sylow $s$-subgroup of $K$ by Lemma \ref{l:frob}. Hence $2s\in\pi(G)$, and the proof of Theorem \ref{t:main} is complete.
\end{proof}

Now we are able to prove Theorem \ref{t:cyclic}. Let $S=H/K$. Clearly, we may assume that $\Out S$ is not cyclic. In particular, we may assume that $S$ is neither sporadic nor alternating with the following convention: if $S=Alt_6\simeq L_2(9)$, we regard $S$ as a group of Lie type.

If $L$ is sporadic and $L\neq J_2$, or if $L=Alt_n$ and $n\neq 6,10$, then $G\simeq L$ (see \cite{98MazShi} and \cite{13Gor.t} respectively). If $L=J_2$, then $G\simeq L$ or $S=Alt_8$ by \cite{98MazShi}. If $L=Alt_{10}$, then $G\simeq L$ or $S=Alt_5$ by \cite{10Sta.t}. If $L=Alt_6$, we regard $L$ as a group of Lie type.

Let $L$ be a group of Lie type. By \cite[Theorem 1]{20YanGrVas}, it follows that $K$ is nilpotent, and so $G$ satisfies the hypothesis of Theorem \ref{t:main}.
If $G/H$ is not cyclic or if $S\neq L_2(q)$ and $G/H$ contains a diagonal automorphism of $S$, then there is $r\in \pi(G)$ adjacent to all  other vertices in $GK(G)$. But this is impossible by Lemma \ref{l:Lie_type} since $GK(G)=GK(L)$. This contradiction completes the proof of Theorem \ref{t:cyclic}.

\section{Groups almost recognizable by prime graph}\label{s:graph}

Given a positive integer $k$, a finite group $G$ is said to be  $k$-recognizable by prime graph if there are exactly $k$ pairwise  nonisomorphic finite groups $H$ with $GK(H)=GK(G)$ and almost recognizable by prime graph if it is $k$-recognizable for some $k$.

By \cite[Theorem 1.3]{21CamMas}, if $G$ is almost recognizable by prime graph, then $G$ is almost simple and each group $H$ with $GK(H)=GK(G)$ is almost simple. So if $G$ is a $k$-recognizable group, then $k$ is at most the number of almost simple groups $H$ such that $\pi(H)=\pi(G)$.
By \cite[Proposition 4.2]{21CamMas}, this number is at most $O(|\pi(G)|^7)$. A direct corollary of this discussion is the following theorem.

\begin{theorem*}[{\cite[Theorem 1.4]{21CamMas}}]
 There exists a function $F(x) = O(x^7)$ such that for each labeled graph $\Gamma$, the following conditions are equivalent:
\begin{enumerate}
 \item there exist infinitely many groups $H$ such that $GK(H) = \Gamma$;
 \item there exist more then $F(|V(\Gamma)|)$ groups $H$ such that $GK(H) = \Gamma$, where $V(\Gamma)$ is the set of the vertices of $\Gamma$.
\end{enumerate}
\end{theorem*}

It is clear that estimating $k$ we do not need to calculate all almost simple groups $H$ such that $\pi(H)=\pi(G)$. It is sufficient to calculate those $H$ whose prime graph satisfies some necessary conditions for $H$ to be almost recognizable by prime graph. One of these conditions is stated in \cite[Theorem 1.3]{21CamMas}: 2 is nonadjacent to at least one odd prime in $GK(H)$. But in fact this condition can be strengthened: every $r\in\pi(H)$ is nonadjacent to at least one prime $s\neq r$ in $GK(H)$. Indeed, otherwise $GK(H)=GK(H\times \mathbb Z^k_r)$ for all positive integers $k$. Applying  Theorem \ref{t:main}, we see that it sufficient to calculate $H$ such that $H/S$  is cyclic, where $S$ is the socle of $H$.

\begin{lemma}\label{l:est} There is a function $F(x)=O(x^2)$ such that if $S$ is a finite simple group of Lie type, then   there are at most $F(|\pi(S)|)$ almost simple groups $H$ with socle $S$ such that $H/S$ is cyclic.
\end{lemma}

\begin{proof}
Let $n$ be the Lie rank of $S$ and $q=p^l$ the order of the base field of $S$. Denote the number of divisors of $l$ by $d(l)$. By \cite[Lemma 2.7]{21CamMas},  we have $n\leqslant 2|\pi(S)|+3$ and $d(l)\leqslant |\pi(S)|+1$.

Steinberg's theorem \cite[Theorem 2.5.12]{98GorLySol} states that $\Out S=\Outdiag S\rtimes \Phi_S\Gamma_S$, where $|\Outdiag S|\leqslant n+1$ and $\Phi_S\Gamma_S$ is either a subgroup in $\mathbb{Z}_l\times Sym_3$ or a cyclic group of order $2l$ or $3l$. In any case the number of cyclic subgroups of $\Phi_S\Gamma_S$ is at most $6d(l)$. Thus the number of cyclic subgroups of $\Out S$ is at most $6(n+1)d(l)$, which is $O(|\pi(S)|^2)$ by the preceding paragraph.
\end{proof}

Now we are ready to prove Theorem \ref{t:poly} (in fact we follow the lines of the proof of \cite[Theorem 1.4]{21CamMas} but Theorem \ref{t:main} allows us to use the bound of Lemma \ref{l:est} instead of that of \cite[Proposition 4.1]{21CamMas}). It is sufficient to show that there exists a function $F(x)=O(x^5)$ such that for every finite group $G$,
if $G$ is almost recognizable by prime graph, then there are at most $F (|\pi(G)|)$ pairwise nonisomorphic groups $H$ with $GK(H) = GK(G)$.

Assume that $G$ is $k$-recognizable by prime graph. By \cite[Theorem 1.3]{21CamMas}, each group~$H$ with $GK(H) = GK(G)$ is almost simple. Furthermore, as we remarked, every $r\in\pi(H)$ is nonadjacent to at least one prime $s\neq r$ in $GK(H)$. By Theorem \ref{t:main}, it follows that $H$ is a cyclic extension of its socle.
By \cite[Proposition 3.1]{21CamMas}, the number of nonabelian simple groups $S$ such that $\pi(S)\subseteq \pi(G)$ is bounded by $F_1(|\pi(G)|)$ with $F_1(x)=O(x^3)$. Applying Lemma \ref{l:est}, we see that the number of almost simple groups $H$ with socle $S$ such that $H/S$ is cyclic is at most $F_2(|\pi(S)|)$, where $F_2(x)=O(x^3)$. Thus $k\leqslant F_1(|\pi(G)|)F_2(|\pi(G)|)=O(|\pi(G)|^5)$, and this completes the proof of Theorem \ref{t:poly}.

\end{document}